\documentclass[a4paper,11pt,fleqn]{article}
\usepackage[obeyspaces,hyphens,spaces]{url}
\usepackage[dvipsnames]{xcolor}
\usepackage{etoolbox}
\usepackage{amsmath}
\usepackage{amssymb}
\usepackage{euscript}
\usepackage{mathrsfs}  
\usepackage{pifont}    
\usepackage{mathtools}
\usepackage{srcltx}    
\usepackage{charter}
\allowdisplaybreaks 
\definecolor{labelkey}{HTML}{0455BF}
\definecolor{refkey}{rgb}{0,0.6,0.0}
\definecolor{dblue}{HTML}{0455BF}
\definecolor{dgreen}{HTML}{02724A}
\definecolor{myellow}{HTML}{D97904}
\definecolor{dred}{HTML}{D90404}
\usepackage{upref} 
\usepackage{hyperref}
\hypersetup{colorlinks=true,linktocpage=true,linkcolor=dblue,%
citecolor=dgreen,urlcolor=dred}
\renewcommand{\leq}{\ensuremath{\leqslant}}
\renewcommand{\geq}{\ensuremath{\geqslant}}

\newcommand{\Argmin}{\ensuremath{\text{\rm Argmin}\,}}

\newcommand{\minmax}[3]{\ensuremath{%
\underset{\substack{#1}}{\text{\rm minimize}}\;\:
\underset{\substack{#2}}{\text{\rm maximize}}\;\;#3}}
\newcommand{\minimize}[2]{\ensuremath{\underset{\substack{{#1}}}%
{\text{\rm minimize}}\;\;#2}}
 
\newcommand{\Scal}[2]{\bigg\langle{#1}\;\bigg|\:{#2}\bigg\rangle} 
 
\newcommand{\scal}[2]{{\langle{{#1}\mid{#2}}\rangle}}
\newcommand{\sscal}[2]{{\big\langle{{#1}\mid{#2}}\big\rangle}}

\newcommand{\menge}[2]{\big\{{#1}~|~{#2}\big\}} 
\newcommand{\Menge}[2]{\left\{{#1}~\middle|~{#2}\right\}} 
\newcommand{\lag}{\ensuremath{\boldsymbol{M}}}

\newcommand{\GGG}{\ensuremath{\boldsymbol{\mathcal{G}}}}
\newcommand{\HHH}{\ensuremath{\boldsymbol{\mathcal{H}}}}
\newcommand{\UUU}{\ensuremath{\boldsymbol{\mathcal{U}}}}
\newcommand{\VVV}{\ensuremath{\boldsymbol{\mathcal{V}}}}
\newcommand{\KKK}{\ensuremath{\boldsymbol{\mathcal{K}}}}

\newcommand{\XXX}{\ensuremath{\boldsymbol{\mathcal{X}}}}
\newcommand{\HH}{\ensuremath{{\mathcal{H}}}}

\newcommand{\XX}{\ensuremath{{\mathcal{X}}}}

\newcommand{\GG}{\ensuremath{{\mathcal{G}}}}

\newcommand{\NN}{\ensuremath{\mathbb{N}}}

\newcommand{\pnabla}[1]{\ensuremath{\nabla_{\!#1}}}
\newcommand{\Sum}{\ensuremath{\displaystyle\sum}}

\newcommand{\emp}{\ensuremath{\varnothing}}

\newcommand{\Id}{\ensuremath{\mathrm{Id}}}
\newcommand{\BI}{\ensuremath{\mathbb{I}}}
\newcommand{\RR}{\ensuremath{\mathbb{R}}}
\newcommand{\RP}{\ensuremath{\left[0,{+}\infty\right[}}

\newcommand{\RPP}{\ensuremath{\left]0,{+}\infty\right[}}

\newcommand{\RPX}{\ensuremath{\left[0,{+}\infty\right]}}
\newcommand{\RX}{\ensuremath{\left]{-}\infty,{+}\infty\right]}}

\newcommand{\weakly}{\ensuremath{\rightharpoonup}}
\newcommand{\exi}{\ensuremath{\exists\,}}

\newcommand{\ran}{\ensuremath{\text{\rm ran}\,}}

\newcommand{\zer}{\text{\rm zer}\,}
\newcommand{\pinf}{\ensuremath{{{+}\infty}}}

\newcommand{\dom}{\ensuremath{\text{\rm dom}\,}}

\newcommand{\prox}{\ensuremath{\text{\rm prox}}}

\newcommand{\gra}{\ensuremath{\text{\rm gra}\,}}

\newcommand{\bdry}{\ensuremath{\text{\rm bdry}\,}}

\newcommand{\sri}{\ensuremath{\text{\rm sri}\,}}

\newcommand{\zeroun}{\ensuremath{\left]0,1\right[}}

\def\abstract{\noindent{\bfseries Abstract}. \ignorespaces}

\usepackage{theorem}
\newtheorem{theorem}{Theorem}[section]

\newtheorem{proposition}[theorem]{Proposition}

\theoremstyle{plain}{\theorembodyfont{\rmfamily}%
}

\theoremstyle{plain}{\theorembodyfont{\rmfamily}%
\newtheorem{example}[theorem]{Example}}
\theoremstyle{plain}{\theorembodyfont{\rmfamily}%
\newtheorem{remark}[theorem]{Remark}}
\theoremstyle{plain}{\theorembodyfont{\rmfamily}%
}
\theoremstyle{plain}{\theorembodyfont{\rmfamily}%
}
\theoremstyle{plain}{\theorembodyfont{\rmfamily}%
}
\theoremstyle{plain}{\theorembodyfont{\rmfamily}%
\newtheorem{problem}[theorem]{Problem}}
\theoremstyle{plain}{\theorembodyfont{\rmfamily}%
}
\usepackage{enumitem}
\setlist[enumerate]{itemsep=2pt}
\setlist[itemize]{itemsep=2pt}

\numberwithin{equation}{section}
\newcommand*\mute{{\mkern 2mu\cdot\mkern 2mu}}
\usepackage{authblk}

\newcommand{\email}[1]{\href{mailto:#1}{\nolinkurl{#1}}}
\author{Minh N. B\`ui}
\author{Patrick L. Combettes}
\affil{North Carolina State University,
Department of Mathematics, Raleigh, NC 27695-8205, USA\\
\email{mnbui@ncsu.edu} and \email{plc@math.ncsu.edu}
}

\begin{document}

\title{\sffamily\huge%
Analysis and Numerical Solution of a\\ Modular
Convex Nash Equilibrium Problem\thanks{%
Contact author: P. L. Combettes.
Email: \email{plc@math.ncsu.edu}.
Phone: +1 919 515 2671.
This work was supported by the National Science
Foundation under grant DMS-1818946.}
}

\date{~}

\maketitle

\begin{abstract}
We investigate a modular convex Nash equilibrium problem involving
nonsmooth functions acting on linear mixtures of strategies, as
well as smooth coupling functions. An asynchronous block-iterative
decomposition method is proposed to solve it.
\end{abstract}

\vskip 12mm

\section{Introduction}
\label{sec:1}

We consider a noncooperative game with $p$ players indexed by
$I=\{1,\ldots,p\}$, in which the strategy $x_i$ of player 
$i\in I$ lies in a real Hilbert space $\HH_i$. A strategy 
profile is a point $\boldsymbol{x}=(x_i)_{i\in I}$ in the Hilbert
direct sum $\HHH=\bigoplus_{i\in I}\HH_i$, and the associated
profile of the players other than $i\in I$ is the vector
$\boldsymbol{x}_{\smallsetminus i}=(x_j)_{j\in
I\smallsetminus\{i\}}$ in $\HHH_{\smallsetminus i}=
\bigoplus_{j\in I \smallsetminus\{i\}}\HH_j$. For every $i\in I$ 
and every $(x_i,\boldsymbol{y})\in\HH_i\times\HHH$, we set
$(x_i;\boldsymbol{y}_{\smallsetminus i})=(y_1,\ldots,y_{i-1},x_i,
y_{i+1},\ldots,y_p)$. Given a real Hilbert space $\HH$, we denote
by $\Gamma_0(\HH)$ the class of lower semicontinuous convex
functions $\varphi\colon\HH\to\RX$ which are proper in the sense
that $\dom\varphi=\menge{x\in\HH}{\varphi(x)<\pinf}\neq\emp$.

A fundamental equilibrium notion was introduced by Nash in
\cite{Nash50,Nash51} to describe a state in which the loss of each
player cannot be reduced by unilateral deviation.
A general formulation of the Nash equilibrium problem is
\begin{equation}
\label{e:0}
\text{find}\;\:\boldsymbol{x}
\in\HHH\;\:\text{such that}\;\:
(\forall i\in I)\;\;{x}_i\in\Argmin
{\boldsymbol{\ell}_i(\mute;
\boldsymbol{x}_{\smallsetminus i})},
\end{equation}
where $\boldsymbol{\ell}_i\colon\HHH\to\RX$ is the global loss
function of player $i\in I$. We make the following
assumption: for every $i\in I$ and every
$\boldsymbol{x}\in\HHH$, the function
$\boldsymbol{\ell}_i(\mute;\boldsymbol{x}_{\smallsetminus i})$ is
convex. Such convex Nash equilibrium problems have been
studied since the early 1970s \cite{Bens72}; see
\cite{Atto08,Boer21,Borz13,Bric13,Cohe87,Sign21,Comi12,%
Facc07,Gaut21,Heus09,Sori16} for further work.
We consider the following modular formulation of \eqref{e:0},
wherein the functions $(\boldsymbol{\ell}_i)_{i\in I}$ are
decomposed into elementary components. This decomposition will
provide more modeling flexibility and lead to efficient
solution methods.

\begin{problem}
\label{prob:1}
Let $(\HH_i)_{i\in I}$ and $(\GG_k)_{k\in K}$
be finite families of real Hilbert spaces, and set
$\HHH=\bigoplus_{i\in I}\HH_i$ and
$\GGG=\bigoplus_{k\in K}\GG_k$.
Suppose that the following are satisfied:
\begin{enumerate}[label={\rm[\alph*]}]
\item
\label{prob:1a}
For every $i\in I$, $\varphi_i\in\Gamma_0(\HH_i)$.
\item
\label{prob:1b}
For every $i\in I$, $\boldsymbol{f}_{\!i}\colon\HHH\to\RR$ is such
that, for every $\boldsymbol{x}\in\HHH$, $\boldsymbol{f}_{\!i}
(\mute;\boldsymbol{x}_{\smallsetminus i})\colon\HH_i\to\RR$ is
convex and differentiable, and we denote its gradient at $x_i$ by
$\pnabla{i}\boldsymbol{f}_{\!i}(\boldsymbol{x})$. Further, the
operator
$\boldsymbol{G}\colon\HHH\to\HHH\colon\boldsymbol{x}\mapsto
(\pnabla{i}\boldsymbol{f}_{\!i}(\boldsymbol{x}))_{i\in I}$ is
monotone and Lipschitzian.
\item
\label{prob:1c}
For every $k\in K$, $g_k\in\Gamma_0(\GG_k)$ and
$\boldsymbol{L}_k\colon\HHH\to\GG_k$ is linear and bounded.
\end{enumerate}
The goal is to
\begin{equation}
\label{e:nash}
\text{find}\;\:\boldsymbol{x}\in\HHH\;\:
\text{such that}\;\:(\forall i\in I)\;\;
x_i\in\Argmin{\varphi_i+\boldsymbol{f}_{\!i}(\mute;
\boldsymbol{x}_{\smallsetminus i})
+\sum_{k\in K}(g_k\circ\boldsymbol{L}_k)(\mute;
\boldsymbol{x}_{\smallsetminus i})}.
\end{equation}
\end{problem}

In Problem~\ref{prob:1}, the individual loss of player $i\in I$ is
a nondifferentiable function $\varphi_i$, while his joint loss is
decomposed into a differentiable function $\boldsymbol{f}_{\!i}$
and a sum of nonsmooth functions $(g_k)_{k\in K}$ acting on linear
mixtures of the strategies. To the best of our knowledge, such a
general formulation of a convex Nash equilibrium has not been
considered in the literature. As will be seen in
Section~\ref{sec:2}, it constitutes a flexible framework that
subsumes a variety of existing equilibrium models. In
Section~\ref{sec:3}, we embed Problem~\ref{prob:1} in an 
inclusion problem in the bigger space $\HHH\oplus\GGG$, and we 
employ the new problem to provide conditions for the existence of
solutions to \eqref{e:nash}. This embedding is also exploited in
Section~\ref{sec:4} to devise an asynchronous block-iterative
algorithm to solve Problem~\ref{prob:1}. The proposed method
features several innovations that are particularly relevant in
large-scale problems: first, each function and each linear operator
in \eqref{e:nash} is activated separately; second, only a subgroup
of functions needs to be activated at any iteration; third, the
computations are asynchronous in the sense that the result of
calculations initiated at earlier iterations can be incorporated at
the current one.

\section{Notation}
\label{sec:2-}
General background on monotone operators and related notions can
be found in \cite{Livre1}. 
Let $\HH$ be a real Hilbert space.
We denote by $2^{\HH}$ the power set of $\HH$ and by $\Id$ the
identity operator on $\HH$. Let $A\colon\HH\to 2^{\HH}$. The
domain of $A$ is $\dom A=\menge{x\in\HH}{Ax\neq\emp}$, the range
of $A$ is $\ran A=\bigcup_{x\in\dom A}Ax$, the graph of $A$ is
$\gra A=\menge{(x,x^*)\in\HH\times\HH}{x^*\in Ax}$, the set of
zeros of $A$ is $\zer A=\menge{x\in\HH}{0\in Ax}$, the inverse
of $A$ is $A^{-1}\colon\HH\to 2^{\HH}\colon x^*\mapsto
\menge{x\in\HH}{x^*\in Ax}$, and the resolvent of $A$ is
$J_A=(\Id+A)^{-1}$. Now suppose that $A$ is monotone, that is, 
\begin{equation}
\big(\forall(x,x^*)\in\gra A\big)
\big(\forall(y,y^*)\in\gra A\big)\quad
\scal{x-y}{x^*-y^*}\geq 0.
\end{equation}
Then $A$ is maximally monotone if, for every
monotone operator $\widetilde{A}\colon\HH\to 2^{\HH}$,
$\gra A\subset\gra\widetilde{A}$ $\Rightarrow$
$A=\widetilde{A}$; $A$ is strongly monotone with constant
$\alpha\in\RPP$ if $A-\alpha\Id$ is monotone; and $A$ is $3^*$
monotone if 
\begin{equation}
(\forall x\in\dom A)(\forall x^*\in\ran A)\quad
\sup_{(y,y^*)\in\gra A}\scal{x-y}{y^*-x^*}<\pinf.
\end{equation}
Let $\varphi\in\Gamma_0(\HH)$. Then $\varphi$ is supercoercive if
$\lim_{\|x\|\to\pinf}\varphi(x)/\|x\|=\pinf$
and uniformly convex if there exists an increasing function
$\phi\colon\RP\to\RPX$ that vanishes only at $0$ such that
\begin{multline}
(\forall x\in\dom\varphi)(\forall y\in\dom\varphi)
(\forall\alpha\in\zeroun)\\
\varphi\big(\alpha x+(1-\alpha)y\big)+\alpha(1-\alpha)
\phi\big(\|x-y\|\big)\leq\alpha\varphi(x)
+(1-\alpha)\varphi(y).
\end{multline}
For every $x\in\HH$, $\prox_\varphi x$ denotes the unique
minimizer of $\varphi+(1/2)\|\mute-x\|^2$. The subdifferential of
$\varphi$ is the maximally monotone operator
\begin{equation}
\partial\varphi\colon\HH\to 2^\HH\colon
x\mapsto\menge{x^*\in\HH}{(\forall y\in\HH)\;\;
\scal{y-x}{x^*}+\varphi(x)\leq\varphi(y)}.
\end{equation}
Finally, given a nonempty convex subset $C$ of $\HH$,
the indicator function of $C$ is
\begin{equation}
\iota_C\colon\HH\to\RPX\colon x\mapsto
\begin{cases}
0,&\text{if}\;\:x\in C;\\
\pinf,&\text{otherwise,}
\end{cases}
\end{equation}
and the strong relative interior of $C$ is
\begin{equation}
\sri C=\Menge{x\in C}{\bigcup_{\lambda\in\,\RPP}\lambda(C-x)
\;\:\text{is a closed vector subspace of}\;\:\HH}.
\end{equation}

\section{Instantiations of Problem~\ref{prob:1}}
\label{sec:2}

Throughout this section, $\HH$ is a real Hilbert space.
We illustrate the wide span of Problem~\ref{prob:1} by
showing that common formulations encountered in
various fields can be recast as special cases of it.

\begin{example}[quadratic coupling]
\label{ex:2}
Let $I$ be a nonempty finite set. For every $i\in I$,
let $\varphi_i\in\Gamma_0(\HH)$,
let $\Lambda_i$ be a nonempty finite set,
let $(\omega_{i,\ell,j})_{\ell\in\Lambda_i,j\in
I\smallsetminus\{i\}}$ be in $\RP$,
and let $(\kappa_{i,\ell})_{\ell\in\Lambda_i}$ be in $\RPP$.
Additionally, set $\HHH=\bigoplus_{i\in I}\HH$.
The problem is to
\begin{equation}
\label{e:nash4}
\text{find}\;\:\boldsymbol{x}\in\HHH\;\:
\text{such that}\;\:(\forall i\in I)\;\;{x}_i\in
\Argmin{\varphi_i+
\sum_{\ell\in\Lambda_i}\frac{\kappa_{i,\ell}}{2}\Bigg\|
\mute-\sum_{j\in I\smallsetminus \{i\}}
\omega_{i,\ell,j}{x}_j\Bigg\|^2}.
\end{equation}
It is assumed that
\begin{equation}
\label{e:2096}
(\forall\boldsymbol{x}\in\HHH)(\forall\boldsymbol{y}\in\HHH)\quad
\sum_{i\in I}\sum_{\ell\in\Lambda_i}\kappa_{i,\ell}
\Scal{x_i-y_i}{x_i-y_i-\sum_{j\in I\smallsetminus\{i\}}
\omega_{i,\ell,j}(x_j-y_j)}\geq 0.
\end{equation}
Define
\begin{equation}
\label{e:2287}
(\forall i\in I)\quad\boldsymbol{f}_{\!i}\colon\HHH\to\RR\colon
\boldsymbol{x}\mapsto
\sum_{\ell\in\Lambda_i}\frac{\kappa_{i,\ell}}{2}\Bigg\|
x_i-\sum_{j\in I\smallsetminus\{i\}}\omega_{i,\ell,j}x_j\Bigg\|^2.
\end{equation}
Then, for every $i\in I$ and every $\boldsymbol{x}\in\HHH$,
$\boldsymbol{f}_{\!i}(\mute;\boldsymbol{x}_{\smallsetminus i})$
is convex and differentiable with
\begin{equation}
\pnabla{i}\boldsymbol{f}_{\!i}(\boldsymbol{x})
=\sum_{\ell\in\Lambda_i}\kappa_{i,\ell}\Bigg(
x_i-\sum_{j\in I\smallsetminus\{i\}}\omega_{i,\ell,j}x_j\Bigg).
\end{equation}
Hence, in view of \eqref{e:2096}, the operator
$\boldsymbol{G}\colon\HHH\to\HHH\colon
\boldsymbol{x}\mapsto
(\pnabla{i}\boldsymbol{f}_{\!i}(\boldsymbol{x}))_{i\in I}$
is monotone and Lipschitzian.
Thus, \eqref{e:nash4} is a special case of \eqref{e:nash}
with $K=\emp$ and $(\forall i\in I)$ $\HH_i=\HH$.
This scenario unifies models found in \cite{Acke80,Baus21,Sign21}.
\end{example}

\begin{example}
\label{ex:23}
In \eqref{e:nash4}, suppose that, for every $i\in I$, 
$C_i$ is a nonempty closed convex subset of $\HH_i$,
$\varphi_i=\iota_{C_i}$, $\Lambda_i=\{1\}$, and
$\kappa_{i,1}=1$. Then \eqref{e:nash4} becomes
\begin{equation}
\label{e:nash14}
\text{find}\;\:\boldsymbol{x}\in\HHH\;\:
\text{such that}\;\:(\forall i\in I)\;\;{x}_i\in
\Argmin_{\!C_i}{\Bigg\|\mute-\sum_{j\in I\smallsetminus \{i\}}
\omega_{i,1,j}{x}_j\Bigg\|^2}.
\end{equation}
In addition, \eqref{e:2096} is satisfied when 
\begin{equation}
\label{e:h4}
\begin{cases}
(\forall i\in I)\;\;\sum_{j\in I\smallsetminus\{i\}}
\omega_{i,1,j}\leq 1\\
(\forall j\in I)\;\;\sum_{i\in I\smallsetminus\{j\}}
\omega_{i,1,j}\leq 1,
\end{cases}
\end{equation}
which places us in the setting of Example~\ref{ex:2}. The
formulation \eqref{e:nash14}--\eqref{e:h4} unifies models found in
\cite{Bail12}.
\end{example}

\begin{example}[minimax]
\label{ex:4}
Let $I$ be a finite set and suppose that $\emp\neq J\subset I$.
Let $(\HH_i)_{i\in I}$ be real Hilbert spaces, and set
$\UUU=\bigoplus_{i\in I\smallsetminus J}\HH_i$ and
$\VVV=\bigoplus_{j\in J}\HH_j$.
For every $i\in I$, let $\varphi_i\in\Gamma_0(\HH_i)$.
Further, let
$\boldsymbol{\EuScript{L}}\colon\UUU\oplus\VVV\to\RR$ be
differentiable with a Lipschitzian gradient and 
such that, for every $\boldsymbol{u}\in\UUU$ and every
$\boldsymbol{v}\in\VVV$, the functions
${-}\boldsymbol{\EuScript{L}}(\boldsymbol{u},\mute)$ and 
$\boldsymbol{\EuScript{L}}(\mute,\boldsymbol{v})$ are
convex. Consider the multivariate minimax problem
\begin{equation}
\label{e:mima}
\minmax{\boldsymbol{u}\in\UUU}{\boldsymbol{v}\in\VVV}
{\sum_{i\in I\smallsetminus J}\varphi_i(u_i)
+\boldsymbol{\EuScript{L}}(\boldsymbol{u},\boldsymbol{v})
-\sum_{j\in J}\varphi_j(v_j)}.
\end{equation}
Now set $\HHH=\UUU\oplus\VVV$ and define
\begin{equation}
(\forall i\in I)\quad
\boldsymbol{f}_{\!i}\colon\HHH\to\RR\colon
(\boldsymbol{u},\boldsymbol{v})\mapsto
\begin{cases}
\boldsymbol{\EuScript{L}}(\boldsymbol{u},\boldsymbol{v}),
&\text{if}\;\:i\in I\smallsetminus J;\\
{-}\boldsymbol{\EuScript{L}}(\boldsymbol{u},\boldsymbol{v}),
&\text{if}\;\:i\in J.
\end{cases}
\end{equation}
Then $\HHH=\bigoplus_{i\in I}\HH_i$ and \eqref{e:mima} can be put
in the form
\begin{equation}
\text{find}\;\:\boldsymbol{x}\in\HHH\;\:
\text{such that}\;\:(\forall i\in I)\;\;
{x}_i\in\Argmin{\varphi_i+\boldsymbol{f}_{\!i}(\mute;
\boldsymbol{x}_{\smallsetminus i})}.
\end{equation}
Let us verify Problem~\ref{prob:1}\ref{prob:1b}.
On the one hand, we have
\begin{equation}
\label{e:78}
(\forall i\in I)(\forall\boldsymbol{x}\in\HHH)\quad
\pnabla{i}\boldsymbol{f}_{\!i}(\boldsymbol{x})=
\begin{cases}
\pnabla{i}\boldsymbol{\EuScript{L}}(\boldsymbol{x}),
&\text{if}\;\:i\in I\smallsetminus J;\\
{-}\pnabla{i}\boldsymbol{\EuScript{L}}(\boldsymbol{x}),
&\text{if}\;\:i\in J.
\end{cases}
\end{equation}
Hence, the operator
\begin{equation}
\boldsymbol{G}\colon\HHH\to\HHH\colon
\boldsymbol{x}\mapsto
\big(\pnabla{i}\boldsymbol{f}_{\!i}(\boldsymbol{x})\big)_{i\in I}=
\Big(\big(\pnabla{i}\boldsymbol{\EuScript{L}}(\boldsymbol{x})
\big)_{i\in I\smallsetminus J},
\big({-}\pnabla{j}\boldsymbol{\EuScript{L}}
(\boldsymbol{x})\big)_{j\in J}\Big)
\end{equation}
is monotone \cite{Rock70,Roc71d} and Lipschitzian. Consequently,
\eqref{e:mima} is an instantiation of \eqref{e:nash}. Special cases
of \eqref{e:mima} under the above assumptions can be found in 
\cite{Sign21,Mont15,Nemi04,Rock95,Thek19,Yoon21}.
\end{example}

\begin{example}[``generalized'' Nash equilibria]
\label{ex:28}
Consider the setting of Problem~\ref{prob:1} where \ref{prob:1a} 
and \ref{prob:1c} are respectively specialized to
\begin{enumerate}[label={\rm[\alph*']}]
\item
For every $i\in I$, $\varphi_i=\iota_{C_i}$, where
$C_i$ is a nonempty closed convex subset of $\HH_i$.
\setcounter{enumi}{2}
\item
$K=\{1\}$ and $g_1=\iota_{D_1}$,
where $D_1$ is a nonempty closed convex subset of $\GG_1$.
\end{enumerate}
Then \eqref{e:nash} reduces to
\begin{equation}
\label{e:gn}
\text{find}\;\:\boldsymbol{x}\in\HHH\;\:
\text{such that}\;\:(\forall i\in I)\;\;
{x}_i\in\Argmin_{C_i}{\;\boldsymbol{f}_{\!i}(\mute;
\boldsymbol{x}_{\smallsetminus i})+
(\iota_{D_1}\circ\boldsymbol{L}_1)(\mute;
\boldsymbol{x}_{\smallsetminus i})}.
\end{equation}
This formulation is often referred to as a generalized Nash 
equilibrium; see, e.g., \cite{Facc07,Heus09,Kanz19}.
However, as noted in \cite{Rock11}, it is really a standard Nash
equilibrium in the sense of \eqref{e:0} since functions are allowed
to take the value $\pinf$.
\end{example}

\begin{example}[PDE model]
\label{ex:29}
Let $\Omega$ be a nonempty open bounded subset of $\RR^N$.
In Example~\ref{ex:28}, suppose that, for every $i\in I$,
$\HH_i=L^2(\Omega)$. Let $z\in L^2(\Omega)$,
let $(\Omega_i)_{i\in I}$ be nonempty open subsets of $\Omega$
with characteristic functions $(1_{\Omega_i})_{i\in I}$, and, for
every $\boldsymbol{x}\in\HHH$, let $\boldsymbol{S}\boldsymbol{x}$
be the unique weak solution in $H_0^1(\Omega)$ of the Dirichlet
boundary value problem \cite[Chapter~IV.2.1]{Ekel74}
\begin{equation}
\label{e:bdrys}
\begin{cases}
-\Delta y=z+\Sum_{i\in I}1_{\Omega_i}x_i,&\text{on}\;\:\Omega;\\
y=0,&\text{on}\;\:\bdry\Omega.
\end{cases}
\end{equation}
For every $i\in I$, let $r_i\in\HH_i$, let $\alpha_i\in\RPP$,
and suppose that
\begin{equation}
\boldsymbol{f}_{\!i}
\colon\boldsymbol{x}\mapsto
\frac{\alpha_i}{2}\|x_i\|_{\HH_i}^2+
\frac{1}{2}\|\boldsymbol{S}\boldsymbol{x}-r_i\|_{\HH_i}^2.
\end{equation}
In addition, suppose that $\GG_1=H_0^1(\Omega)$ and
$\boldsymbol{L}_1=\boldsymbol{S}$. Then we recover frameworks
investigated in \cite{Borz13,Kanz19}.
\end{example}

\begin{example}[multivariate minimization]
\label{ex:3}
Consider the setting of Problem~\ref{prob:1} where \ref{prob:1b} 
and \ref{prob:1c} are respectively specialized to
\begin{enumerate}[label={\rm[\alph*']}]
\setcounter{enumi}{1}
\item
For every $i\in I$, $\boldsymbol{f}_{\!i}=\boldsymbol{f}$,
where $\boldsymbol{f}\colon\KKK\to\RR$ is a differentiable convex 
function such that $\boldsymbol{G}=\nabla\boldsymbol{f}$ is
Lipschitzian. 
\item
For every $k\in K$, $g_k\colon\GG_k\to\RR$ is convex and G\^ateaux
differentiable, and $\boldsymbol{L}_k\colon
\HHH\to\GG_k\colon\boldsymbol{x}\mapsto
\sum_{j\in I}L_{k,j}x_j$ where, for every $j\in I$, 
$L_{k,j}\colon\HH_j\to\GG_k$ is linear and bounded.
\end{enumerate}
Then \eqref{e:nash} reduces to the multivariate minimization
problem 
\begin{equation}
\label{e:6887}
\minimize{\boldsymbol{x}\in\HHH}{\sum_{i\in I}\varphi_i(x_i)+
\boldsymbol{f}(\boldsymbol{x})+
\sum_{k\in K}g_k\Bigg(\sum_{j\in I}L_{k,j}x_j\Bigg)}.
\end{equation}
Instances of this problem are found in 
\cite{Argy12,Atto08,Bric19,Bric09,Jmiv11,Darb20,Habb13}.
\end{example}

\section{Existence of solutions}
\label{sec:3}

Our first existence result revolves around an embedding of
Problem~\ref{prob:1} in a larger inclusion problem in the space
$\HHH\oplus\GGG$.

\begin{proposition}
\label{p:1}
Consider the setting of Problem~\ref{prob:1} and set
$(\forall i\in I)$
$\Pi_i\colon\HHH\to\HH_i\colon\boldsymbol{x}\mapsto x_i$.
Suppose that 
$(\overline{\boldsymbol{x}},\overline{\boldsymbol{v}}^*)
\in\HHH\oplus\GGG$ satisfies
\begin{equation}
\label{e:kt}
\begin{cases}
(\forall i\in I)\;\;
-\pnabla{i}\boldsymbol{f}_{\!i}(\overline{\boldsymbol{x}})-
\Sum_{k\in K}\Pi_i(\boldsymbol{L}_k^*\overline{v}_k^*)
\in\partial\varphi_i(\overline{x}_i)\\
(\forall k\in K)\;\;
\boldsymbol{L}_k\overline{\boldsymbol{x}}
\in\partial g_k^*(\overline{v}_k^*).
\end{cases}
\end{equation}
Then $\overline{\boldsymbol{x}}$ solves \eqref{e:nash}.
\end{proposition}
\begin{proof}
Take $i\in I$ and set
\begin{equation}
\label{e:8920}
f_i=\boldsymbol{f}_{\!i}(\mute;
\overline{\boldsymbol{x}}_{\smallsetminus i}),
\quad\overline{\boldsymbol{s}}_i=(0;
\overline{\boldsymbol{x}}_{\smallsetminus i}),
\quad\text{and}\quad
(\forall k\in K)\;\;
\widetilde{g}_k=(g_k\circ\boldsymbol{L}_k)(\mute;
\overline{\boldsymbol{x}}_{\smallsetminus i}).
\end{equation}
Then, by Problem~\ref{prob:1}\ref{prob:1b},
$f_i\colon\HH_i\to\RR$ is convex and G\^ateaux differentiable,
and $\nabla f_i(\overline{x}_i)=\pnabla{i}\boldsymbol{f}_{\!i}
(\overline{\boldsymbol{x}})$.
At the same time, 
\begin{equation}
(\forall k\in K)(\forall x_i\in\HH_i)\quad
\widetilde{g}_k(x_i)
=(g_k\circ\boldsymbol{L}_k)(\Pi_i^*x_i+
\overline{\boldsymbol{s}}_i)
=g_k(\boldsymbol{L}_k(\Pi_i^*x_i)+
\boldsymbol{L}_k\overline{\boldsymbol{s}}_i)
\end{equation}
and it thus results from 
\cite[Proposition~16.6(ii)]{Livre1} that
\begin{equation}
(\forall k\in K)(\forall x_i\in\HH_i)\quad
(\Pi_i\circ\boldsymbol{L}_k^*)\big(\partial g_k(
\boldsymbol{L}_k(\Pi_i^*x_i)+
\boldsymbol{L}_k\overline{\boldsymbol{s}}_i)\big)
\subset\partial\widetilde{g}_k(x_i).
\end{equation}
In particular,
\begin{equation}
(\forall k\in K)\quad
(\Pi_i\circ\boldsymbol{L}_k^*)\big(\partial g_k(
\boldsymbol{L}_k\overline{\boldsymbol{x}})\big)
=(\Pi_i\circ\boldsymbol{L}_k^*)\Big(\partial g_k\big(
\boldsymbol{L}_k(\Pi_i^*\overline{x}_i)+
\boldsymbol{L}_k\overline{\boldsymbol{s}}_i\big)\Big)
\subset\partial\widetilde{g}_k(\overline{x}_i).
\end{equation}
Hence, we deduce from \eqref{e:kt} and
\cite[Proposition~16.6(ii)]{Livre1} that
\begin{align}
0&\in\partial\varphi_i(\overline{x}_i)+
\pnabla{i}\boldsymbol{f}_{\!i}(\overline{\boldsymbol{x}})+
\sum_{k\in K}\Pi_i(\boldsymbol{L}_k^*\overline{v}_k^*)
\nonumber\\
&\subset\partial\varphi_i(\overline{x}_i)+
\nabla f_i(\overline{x}_i)+
\sum_{k\in K}(\Pi_i\circ\boldsymbol{L}_k^*)\big(
\partial g_k(\boldsymbol{L}_k\overline{\boldsymbol{x}})\big)
\nonumber\\
&\subset\partial\varphi_i(\overline{x}_i)+
\nabla f_i(\overline{x}_i)+
\sum_{k\in K}\partial\widetilde{g}_k(\overline{x}_i)
\nonumber\\
&\subset\partial\bigg(\varphi_i+f_i+
\sum_{k\in K}\widetilde{g}_k\bigg)
(\overline{x}_i).
\end{align}
Consequently, appealing to Fermat's rule
\cite[Theorem~16.3]{Livre1} and \eqref{e:8920}, we arrive at
\begin{equation}
\overline{x}_i\in
\Argmin{\varphi_i+\boldsymbol{f}_{\!i}(\mute;
\overline{\boldsymbol{x}}_{\smallsetminus i})
+\sum_{k\in K}(g_k\circ\boldsymbol{L}_k)(\mute;
\overline{\boldsymbol{x}}_{\smallsetminus i})},
\end{equation}
which completes the proof.
\end{proof}

We are now in a position to provide specific existence conditions.

\begin{proposition}
\label{p:9}
Consider the setting of Problem~\ref{prob:1}, set
\begin{equation}
\label{e:3038}
\boldsymbol{C}=\menge{(\boldsymbol{L}_k\boldsymbol{x}-
y_k)_{k\in K}}{(\forall i\in I)\;\;x_i\in\dom\varphi_i\;\:
\text{and}\;\:(\forall k\in K)\;\;y_k\in\dom g_k},
\end{equation}
and let $\boldsymbol{Z}\subset\HHH\oplus\GGG$ be the set of
solutions to \eqref{e:kt}. Suppose that
$\boldsymbol{0}\in\sri\boldsymbol{C}$ and that one of the following
is satisfied: 
\begin{enumerate}
\item 
\label{p:9a}
For every $i\in I$, one of the following holds:
\begin{enumerate}[label={\rm\arabic*/}]
\item
\label{p:9a1}
$\partial\varphi_i$ is surjective.
\item
\label{p:9a2}
$\varphi_i$ is supercoercive.
\item
\label{p:9a3}
$\dom\varphi_i$ is bounded.
\item
\label{p:9a4}
$\varphi_i$ is uniformly convex.
\end{enumerate}
\item
\label{p:9b}
$\boldsymbol{G}\colon\HHH\to\HHH\colon\boldsymbol{x}\mapsto
(\pnabla{i}\boldsymbol{f}_{\!i}(\boldsymbol{x}))_{i\in I}$ 
is $3^*$ monotone and surjective.
\end{enumerate}
Then $\boldsymbol{Z}\neq\emp$ and Problem~\ref{prob:1}
has a solution.
\end{proposition}
\begin{proof}
Define
\begin{equation}
\label{e:7439}
\begin{cases}
\boldsymbol{A}\colon\HHH\to 2^{\HHH}\colon
\boldsymbol{x}\mapsto\bigtimes_{i\in I}\partial\varphi_i(x_i)\\
\boldsymbol{B}\colon\GGG\to 2^{\GGG}\colon
\boldsymbol{y}\mapsto\bigtimes_{k\in K}\partial g_k(y_k)\\
\boldsymbol{L}\colon\HHH\to\GGG\colon\boldsymbol{x}\mapsto
(\boldsymbol{L}_k\boldsymbol{x})_{k\in K}
\end{cases}
\end{equation}
and
\begin{equation}
\label{e:d4y}
\boldsymbol{T}\colon\HHH\to 2^{\HHH}\colon\boldsymbol{x}\mapsto
\boldsymbol{A}\boldsymbol{x}+\boldsymbol{L}^*\big(
\boldsymbol{B}(\boldsymbol{L}\boldsymbol{x})\big)+
\boldsymbol{G}\boldsymbol{x}.
\end{equation}
Note that the adjoint of $\boldsymbol{L}$ is
\begin{equation}
\label{e:3851}
\boldsymbol{L}^*\colon\GGG\to\HHH\colon\boldsymbol{v}^*\mapsto
\sum_{k\in K}\boldsymbol{L}_k^*\overline{v}_k^*.
\end{equation}
Now suppose that $\overline{\boldsymbol{x}}\in\zer\boldsymbol{T}$.
Then there exists $\overline{\boldsymbol{v}}^*\in
\boldsymbol{B}(\boldsymbol{L}\overline{\boldsymbol{x}})$ such that
${-}\boldsymbol{G}\overline{\boldsymbol{x}}-
\boldsymbol{L}^*\overline{\boldsymbol{v}}^*
\in\boldsymbol{A}\overline{\boldsymbol{x}}$
or, equivalently, by Problem~\ref{prob:1}\ref{prob:1b} and
\eqref{e:3851}, $(\forall i\in I)$
${-}\pnabla{i}\boldsymbol{f}_{\!i}(\overline{\boldsymbol{x}})-
\sum_{k\in K}\Pi_i(\boldsymbol{L}_k^*\overline{v}_k^*)
\in\partial\varphi_i(\overline{x}_i)$.
Further, \eqref{e:7439} yields
$\overline{v}_k^*\in\partial g_k
(\boldsymbol{L}_k\overline{\boldsymbol{x}})$.
Altogether, in view of \eqref{e:kt} and Proposition~\ref{p:1}, 
we have established the implication
\begin{equation}
\zer\boldsymbol{T}\neq\emp
\quad\Rightarrow\quad\boldsymbol{Z}\neq\emp
\quad\Rightarrow\quad\text{Problem~\ref{prob:1} has a solution.}
\end{equation}
Therefore, it suffices to show that
$\zer\boldsymbol{T}\neq\emp$. To do so, define
\begin{equation}
\begin{cases}
\boldsymbol{\varphi}\colon\HHH\to\RX\colon\boldsymbol{x}\mapsto
\sum_{i\in I}\varphi_i(x_i)\\
\boldsymbol{g}\colon\GGG\to\RX\colon\boldsymbol{y}\mapsto
\sum_{k\in K}g_k(y_k)\\
\boldsymbol{Q}=\boldsymbol{A}
+\boldsymbol{L}^*\circ\boldsymbol{B}\circ\boldsymbol{L}.
\end{cases}
\end{equation}
Then, by \eqref{e:7439} and \cite[Proposition~16.9]{Livre1},
$\boldsymbol{A}=\partial\boldsymbol{\varphi}$ and
$\boldsymbol{B}=\partial\boldsymbol{g}$.
In turn, since \eqref{e:3038} and \eqref{e:7439} imply that
$\boldsymbol{0}\in\sri\boldsymbol{C}
=\sri\!(\boldsymbol{L}(\dom\boldsymbol{\varphi})
-\dom\boldsymbol{g})$, we derive from
\cite[Theorem~16.47(i)]{Livre1} that
$\boldsymbol{Q}=\partial(\boldsymbol{\varphi}+
\boldsymbol{g}\circ\boldsymbol{L})$.
Therefore, in view of \cite[Theorem~20.25 and
Example~25.13]{Livre1},
\begin{equation}
\label{e:1185}
\text{$\boldsymbol{A}$, $\boldsymbol{B}$, and
$\boldsymbol{Q}$ are maximally monotone and $3^*$ monotone}.
\end{equation}

\ref{p:9a}:
Fix temporarily $i\in I$. By \cite[Theorem~20.25]{Livre1},
$\partial\varphi_i$ is maximally monotone.
First, if \ref{p:9a}\ref{p:9a2} holds, then
\cite[Corollary~16.30, and Propositions~14.15 and 16.27]{Livre1}
entail that $\ran\partial\varphi_i
=\dom\partial\varphi_i^*=\HH_i$ and, hence,
\ref{p:9a}\ref{p:9a1} holds. Second,
if \ref{p:9a}\ref{p:9a3} holds, then
$\dom\partial\varphi_i\subset\dom\varphi_i$ is bounded
and, therefore, it follows from
\cite[Corollary~21.25]{Livre1} that \ref{p:9a}\ref{p:9a1} holds.
Finally, if \ref{p:9a}\ref{p:9a4} holds, then
\cite[Proposition~17.26(ii)]{Livre1} implies that
\ref{p:9a}\ref{p:9a2} holds and, in turn, that 
\ref{p:9a}\ref{p:9a1} holds.
Altogether, it is enough to show that
\begin{equation}
\big[\;(\forall i\in I)\;\;
\partial\varphi_i\;\:\text{is surjective}\;\big]
\quad\Rightarrow\quad\zer\boldsymbol{T}\neq\emp.
\end{equation}
Assume that the operators $(\partial\varphi_i)_{i\in I}$ are
surjective and set
\begin{equation}
\boldsymbol{P}={-}\boldsymbol{Q}^{-1}
\circ({-}\boldsymbol{\Id})+\boldsymbol{G}^{-1}.
\end{equation}
Then we derive from \eqref{e:7439} that
$\boldsymbol{A}$ is surjective.
On the other hand, \cite[Proposition~6]{Botr07} asserts that
$\boldsymbol{L}^*\circ\boldsymbol{B}\circ\boldsymbol{L}$
is $3^*$ monotone. Hence, \eqref{e:1185} and
\cite[Corollary~25.27(i)]{Livre1} yields
\begin{equation}
\dom\boldsymbol{Q}^{-1}=\ran\boldsymbol{Q}=\HHH.
\end{equation}
In turn, since $\boldsymbol{Q}^{-1}$ and $\boldsymbol{G}^{-1}$
are maximally monotone, \cite[Theorem~25.3]{Livre1} implies that
$\boldsymbol{P}$ is likewise. Furthermore, we observe that
\begin{equation}
\dom\boldsymbol{G}^{-1}\subset\HHH
=\dom\big({-}\boldsymbol{Q}^{-1}\circ({-}\boldsymbol{\Id})\big) 
\end{equation}
and, by virtue of \eqref{e:1185} and 
\cite[Proposition~25.19(i)]{Livre1}, that
$-\boldsymbol{Q}^{-1}\circ({-}\boldsymbol{\Id})$
is $3^*$ monotone. Therefore, since
$\ran\boldsymbol{G}^{-1}=\dom\boldsymbol{G}=\HHH$,
\cite[Corollary~25.27(ii)]{Livre1} entails that
$\boldsymbol{P}$ is surjective and, in turn, that
$\zer\boldsymbol{P}\neq\emp$.
Consequently, \cite[Proposition~26.33(iii)]{Livre1}
asserts that $\zer\boldsymbol{T}\neq\emp$.

\ref{p:9b}:
Since $\boldsymbol{G}$ is maximally monotone and
$\dom\boldsymbol{G}=\HHH$, it results from \eqref{e:1185} and
\cite[Theorem~25.3]{Livre1} that $\boldsymbol{T}=\boldsymbol{Q}
+\boldsymbol{G}$ is maximally monotone. Hence, since
$\boldsymbol{G}$ is surjective, we derive from \eqref{e:1185} and
\cite[Corollary~25.27(i)]{Livre1} that $\boldsymbol{T}$
is surjective and, therefore, that
$\zer\boldsymbol{T}\neq\emp$.
\end{proof}

\begin{remark}
Sufficient conditions for $\boldsymbol{0}\in\sri\boldsymbol{C}$
to hold in Proposition~\ref{p:9} can be found in
\cite[Proposition~5.3]{Siop13}.
\end{remark}

\section{Algorithm}
\label{sec:4}

The main result of this section is the following theorem, where we
introduce an asynchronous block-iterative algorithm to solve
Problem~\ref{prob:1} and prove its convergence. 

\begin{theorem}
\label{t:1}
Consider the setting of Problem~\ref{prob:1} and set
$(\forall i\in I)$
$\Pi_i\colon\HHH\to\HH_i\colon\boldsymbol{x}\mapsto x_i$. Let 
$(\chi_i)_{i\in I}$ be a family in $\RP$ such that
\begin{equation}
\label{e:Glip}
(\forall\boldsymbol{x}\in\HHH)(\forall\boldsymbol{y}\in\HHH)
\quad
\scal{\boldsymbol{x}-\boldsymbol{y}}{
\boldsymbol{G}\boldsymbol{x}-\boldsymbol{G}\boldsymbol{y}}
\leq\sum_{i\in I}\chi_i\|x_i-y_i\|^2,
\end{equation}
let $\alpha\in\RPP$ and $\varepsilon\in\zeroun$ be such that
$1/\varepsilon>\alpha+\max_{i\in I}\chi_i$,
let $(\lambda_n)_{n\in\NN}$ be in
$\left[\varepsilon,2-\varepsilon\right]$, and let $D\in\NN$.
Suppose that the following are satisfied:
\begin{enumerate}[label={\rm[\alph*]}]
\item
\label{t:1z}
There exists 
$(\overline{\boldsymbol{x}},\overline{\boldsymbol{v}}^*)
\in\HHH\oplus\GGG$ such that \eqref{e:kt} holds.
\item
\label{t:1a}
For every $i\in I$, $x_{i,0}\in\HH_i$ and, for every $n\in\NN$,
$\gamma_{i,n}\in\left[\varepsilon,1/(\chi_i+\alpha)\right]$ and
$c_i(n)\in\NN$ satisfies $n-D\leq c_i(n)\leq n$.
\item
\label{t:1b}
For every $k\in K$, $v_{k,0}^*\in\GG_k$ and, for every $n\in\NN$, 
$\mu_{k,n}\in\left[\alpha,1/\varepsilon\right]$ and 
$d_k(n)\in\NN$ satisfies $n-D\leq d_k(n)\leq n$.
\item
\label{t:1c}
$(I_n)_{n\in\NN}$ are nonempty subsets of $I$
and $(K_n)_{n\in\NN}$ are nonempty subsets of $K$ such that
\begin{equation}
\label{e:1039}
I_0=I,\quad K_0=K,\quad\text{and}\quad
(\exi m\in\NN)(\forall n\in\NN)\;\;
\bigcup_{j=n}^{n+m}I_j=I\;\:\text{and}\;\:
\bigcup_{j=n}^{n+m}K_j=K.
\end{equation}
\end{enumerate}
Further, set $\boldsymbol{L}\colon\HHH\to\GGG\colon
\boldsymbol{x}\mapsto(\boldsymbol{L}_k\boldsymbol{x})_{k\in K}$.
Iterate
\begin{equation}
\label{e:1494}
\begin{array}{l}
\text{for}\;n=0,1,\ldots\\
\left\lfloor
\begin{array}{l}
\text{for every}\;i\in I_n\\
\left\lfloor
\begin{array}{l}
x_{i,n}^*=x_{i,c_i(n)}-\gamma_{i,c_i(n)}\Big(
\pnabla{i}\boldsymbol{f}_{\!i}(\boldsymbol{x}_{c_i(n)})+
\sum_{k\in K}\Pi_i\big(\boldsymbol{L}_k^*v_{k,c_i(n)}^*\big)
\Big)\\
a_{i,n}=\prox_{\gamma_{i,c_i(n)}\varphi_i}x_{i,n}^*\\
a_{i,n}^*=\gamma_{i,c_i(n)}^{-1}(x_{i,n}^*-a_{i,n})
\end{array}
\right.
\\
\text{for every}\;i\in I\smallsetminus I_n\\
\left\lfloor
\begin{array}{l}
(a_{i,n},a_{i,n}^*)=(a_{i,n-1},a_{i,n-1}^*)
\end{array}
\right.\\
\text{for every}\;k\in K_n\\
\left\lfloor
\begin{array}{l}
y_{k,n}^*=\mu_{k,d_k(n)}v_{k,d_k(n)}^*+
\boldsymbol{L}_k\boldsymbol{x}_{d_k(n)}\\
b_{k,n}=\prox_{\mu_{k,d_k(n)}g_k}y_{k,n}^*\\
b_{k,n}^*=\mu_{k,d_k(n)}^{-1}(y_{k,n}^*-b_{k,n})
\end{array}
\right.\\
\text{for every}\;k\in K\smallsetminus K_n\\
\left\lfloor
\begin{array}{l}
(b_{k,n},b_{k,n}^*)=(b_{k,n-1},b_{k,n-1}^*)
\end{array}
\right.\\
\boldsymbol{t}_n^*=\boldsymbol{a}_n^*+
\boldsymbol{G}\boldsymbol{a}_n+
\boldsymbol{L}^*\boldsymbol{b}_n^*\\
\boldsymbol{t}_n=\boldsymbol{b}_n-\boldsymbol{L}\boldsymbol{a}_n
\\
\pi_n=
\scal{\boldsymbol{a}_n-\boldsymbol{x}_n}{\boldsymbol{t}_n^*}+
\scal{\boldsymbol{t}_n}{\boldsymbol{b}_n^*-\boldsymbol{v}_n^*}\\
\text{if}\;\pi_n<0\\
\left\lfloor
\begin{array}{l}
\alpha_n=\lambda_n\pi_n/\big(\|\boldsymbol{t}_n\|^2+
\|\boldsymbol{t}_n^*\|^2\big)\\
\boldsymbol{x}_{n+1}=\boldsymbol{x}_n+
\alpha_n\boldsymbol{t}_n^*\\
\boldsymbol{v}_{n+1}^*=\boldsymbol{v}_n^*+\alpha_n
\boldsymbol{t}_n
\end{array}
\right.\\
\text{else}\\
\left\lfloor
\begin{array}{l}
\big(\boldsymbol{x}_{n+1},\boldsymbol{v}_{n+1}^*\big)=
(\boldsymbol{x}_n,\boldsymbol{v}_n^*).
\end{array}
\right.\\[2mm]
\end{array}
\right.
\end{array}
\end{equation}
Then $(\boldsymbol{x}_n)_{n\in\NN}$ converges weakly to a
solution to Problem~\ref{prob:1}. 
\end{theorem}

The salient features of the proposed algorithm are the following:
\begin{itemize}
\item
{\bfseries Decomposition:}
In \eqref{e:1494}, the functions $(\varphi_i)_{i\in I}$ and
$(g_k)_{k\in K}$ are activated separately via their proximity
operators. 
\item
{\bfseries Block-iterative implementation:}
At iteration $n$, we require that only the subfamilies of functions
$(\varphi_i)_{i\in I_n}$ and $(g_k)_{k\in K_n}$ be activated, as
opposed to all of them as in standard splitting methods. To
guarantee convergence, we ask in condition~\ref{t:1c} of
Theorem~\ref{t:1} that each of these functions be activated
frequently enough.
\item
{\bfseries Asynchronous implementation:}
Given $i\in I$ and $k\in K$, the asynchronous character of the
algorithm is materialized by the variables $c_i(n)$ and
$d_k(n)$ which signal when the underlying computations
incorporated at iteration $n$ were initiated. 
Conditions~\ref{t:1a} and \ref{t:1b} of Theorem~\ref{t:1} ask that
the lag between the initiation and the incorporation of such
computations do not exceed $D$ iterations. The synchronous
implementation is obtained when $c_i(n)=n$ and $d_k(n)=n$ in
\eqref{e:1494}. The introduction of asynchronous and
block-iterative techniques in monotone operator splitting were
initiated in \cite{MaPr18}.
\end{itemize}

In order to prove Theorem~\ref{t:1}, we need to establish some
preliminary properties.

\begin{proposition}
\label{p:10}
Let $(\XX_i)_{i\in\BI}$ be a finite family of real Hilbert spaces
with Hilbert direct sum $\XXX=\bigoplus_{i\in\BI}\XX_i$.
For every $i\in\BI$,
let $P_i\colon\XX_i\to 2^{\XX_i}$ be maximally monotone and let
$Q_i\colon\XXX\to\XX_i$.
It is assumed that $\boldsymbol{Q}\colon\XXX\to\XXX\colon
\boldsymbol{x}\mapsto(Q_i\boldsymbol{x})_{i\in\BI}$ is monotone
and Lipschitzian, and that the problem
\begin{equation}
\label{e:9}
\text{find}\;\:\boldsymbol{x}\in\XXX\;\:
\text{such that}\;\:(\forall i\in\BI)\;\;
0\in P_i{x}_i+Q_i\boldsymbol{x}
\end{equation}
has a solution. Let $(\chi_i)_{i\in\BI}$ be a family in
$\RP$ such that
\begin{equation}
\label{e:3021}
(\forall\boldsymbol{x}\in\XXX)
(\forall\boldsymbol{y}\in\XXX)\quad
\scal{\boldsymbol{x}-\boldsymbol{y}}{
\boldsymbol{Q}\boldsymbol{x}-\boldsymbol{Q}\boldsymbol{y}}
\leq\sum_{i\in\BI}\chi_i\|x_i-y_i\|^2,
\end{equation}
let $\alpha\in\RPP$, let $\varepsilon\in\zeroun$ be such
that $1/\varepsilon>\alpha+\max_{i\in\BI}\chi_i$,
and let $D\in\NN$. For every $i\in\BI$, let $x_{i,0}\in\XX_i$ and,
for every $n\in\NN$, let
$\gamma_{i,n}\in\left[\varepsilon,1/(\chi_i+\alpha)\right]$,
let $\lambda_n\in\left[\varepsilon,2-\varepsilon\right]$,
and let $d_i(n)\in\NN$ be such that
\begin{equation}
\label{e:4de2}
n-D\leq d_i(n)\leq n.
\end{equation}
In addition, let $(\BI_n)_{n\in\NN}$ be nonempty subsets of
$\BI$ such that
\begin{equation}
\label{e:1928}
\BI_0=\BI\quad\text{and}\quad
(\exi m\in\NN)(\forall n\in\NN)\;\;\bigcup_{j=n}^{n+m}\BI_j=\BI.
\end{equation}
Iterate
\begin{equation}
\label{e:alg}
\begin{array}{l}
\text{for}\;n=0,1,\ldots\\
\left\lfloor
\begin{array}{l}
\text{for every}\;i\in\BI_n\\
\left\lfloor
\begin{array}{l}
x_{i,n}^*=x_{i,d_i(n)}-
\gamma_{i,d_i(n)}Q_i\boldsymbol{x}_{d_i(n)}\\
p_{i,n}=J_{\gamma_{i,d_i(n)}P_i}x_{i,n}^*\\
p_{i,n}^*=\gamma_{i,d_i(n)}^{-1}(x_{i,n}^*-p_{i,n})
\end{array}
\right.
\\
\text{for every}\;i\in\BI\smallsetminus\BI_n\\
\left\lfloor
\begin{array}{l}
(p_{i,n},p_{i,n}^*)=(p_{i,n-1},p_{i,n-1}^*)
\end{array}
\right.
\\
\boldsymbol{s}_n^*=\boldsymbol{p}_n^*+
\boldsymbol{Q}\boldsymbol{p}_n
\\
\pi_n=\scal{\boldsymbol{p}_n-\boldsymbol{x}_n}{
\boldsymbol{s}_n^*}\\
\text{if}\;\pi_n<0\\
\left\lfloor
\begin{array}{l}
\alpha_n=\lambda_n\pi_n/\|\boldsymbol{s}_n^*\|^2\\
\boldsymbol{x}_{n+1}=\boldsymbol{x}_n+
\alpha_n\boldsymbol{s}_n^*
\end{array}
\right.
\\
\text{else}\\
\left\lfloor
\begin{array}{l}
\boldsymbol{x}_{n+1}=\boldsymbol{x}_n.
\end{array}
\right.\\[2mm]
\end{array}
\right.
\end{array}
\end{equation}
Then the following holds:
\begin{enumerate}
\item
\label{p:10i}
$(\forall i\in\BI)$ $x_{i,n}-p_{i,n}\to 0$.
\item
\label{p:10ii}
$(\boldsymbol{x}_n)_{n\in\NN}$ converges weakly to a solution
to \eqref{e:9}.
\end{enumerate}
\end{proposition}
\begin{proof}
Define
\begin{equation}
\label{e:s}
\lag\colon\XXX\to 2^{\XXX}\colon\boldsymbol{x}\mapsto
\boldsymbol{Q}\boldsymbol{x}+
\bigtimes_{i\in\BI}P_ix_i.
\end{equation}
It follows from \cite[Proposition~20.23]{Livre1} that the operator
$\boldsymbol{x}\mapsto\bigtimes_{i\in\BI}P_ix_i$
is maximally monotone. Thus, since
$\boldsymbol{Q}$ is maximally monotone by
\cite[Corollary~20.28]{Livre1}, we deduce from
\cite[Corollary~25.5(i)]{Livre1} that $\lag$ is maximally monotone.
Further, since \eqref{e:9} has a solution, $\zer\lag\neq\emp$. Set
\begin{equation}
\label{e:2488}
(\forall i\in\BI)(\forall n\in\NN)\quad
\overline{\delta}_i(n)=
\max\menge{j\in\NN}{j\leq n\;\:\text{and}\;\:i\in\BI_j}
\quad\text{and}\quad
\delta_i(n)=d_i\big(\overline{\delta}_i(n)\big),
\end{equation}
and define
\begin{equation}
\label{e:kern}
(\forall n\in\NN)\quad
\boldsymbol{K}_n\colon\XXX\to\XXX\colon\boldsymbol{x}\mapsto
\big(\gamma_{i,\delta_i(n)}^{-1}x_i\big)_{i\in\BI}-
\boldsymbol{Q}\boldsymbol{x}.
\end{equation}
In addition, let $\chi$ be a Lipschitz constant of
$\boldsymbol{Q}$.
Then, the operators $(\boldsymbol{K}_n)_{n\in\NN}$ are
Lipschitzian with constant
$\beta=\sqrt{2(\varepsilon^{-2}+\chi^2)}$.
At the same time, for every $n\in\NN$,
we derive from \eqref{e:kern} and \eqref{e:3021} that
\begin{align}
\label{e:2068}
(\forall\boldsymbol{x}\in\XXX)
(\forall\boldsymbol{y}\in\XXX)\quad
\scal{\boldsymbol{x}-\boldsymbol{y}}{
\boldsymbol{K}_n\boldsymbol{x}-
\boldsymbol{K}_n\boldsymbol{y}}
&=\sum_{i\in\BI}\gamma_{i,\delta_i(n)}^{-1}\|x_i-y_i\|^2-
\scal{\boldsymbol{x}-\boldsymbol{y}}{
\boldsymbol{Q}\boldsymbol{x}-\boldsymbol{Q}\boldsymbol{y}}
\nonumber\\
&\geq\sum_{i\in\BI}(\chi_i+\alpha)\|x_i-y_i\|^2-
\sum_{i\in\BI}\chi_i\|x_i-y_i\|^2
\nonumber\\
&=\alpha\|\boldsymbol{x}-\boldsymbol{y}\|^2
\end{align}
and, in turn, that $\boldsymbol{K}_n$ is $\alpha$-strongly monotone
and maximally monotone \cite[Corollary~20.28]{Livre1}.
Hence, for every $n\in\NN$, \cite[Proposition~22.11(ii)]{Livre1}
implies that there exists $\widetilde{\boldsymbol{x}}_n\in\XXX$
such that
\begin{equation}
\label{e:2138}
\big(\gamma_{i,\delta_i(n)}^{-1}x_{i,\overline{\delta}_i(n)}^*
\big)_{i\in\BI}
=\boldsymbol{K}_n\widetilde{\boldsymbol{x}}_n.
\end{equation}
Therefore, we infer from \eqref{e:alg}, \eqref{e:2488},
\eqref{e:s}, and \eqref{e:kern} that
\begin{align}
(\forall n\in\NN)\quad
\boldsymbol{p}_n
&=\big(p_{i,\overline{\delta}_i(n)}\big)_{i\in\BI}
\label{e:0816}
\\
&=\big(J_{\gamma_{i,\delta_i(n)}P_i}
x_{i,\overline{\delta}_i(n)}^*\big)_{i\in\BI}
\nonumber\\
&=(\boldsymbol{K}_n+\lag)^{-1}\big(
\gamma_{i,\delta_i(n)}^{-1}x_{i,\overline{\delta}_i(n)}^*
\big)_{i\in\BI}
\nonumber\\
&=(\boldsymbol{K}_n+\lag)^{-1}
\big(\boldsymbol{K}_n\widetilde{\boldsymbol{x}}_n\big).
\end{align}
On the other hand, by \eqref{e:alg}, \eqref{e:2488},
\eqref{e:2138}, \eqref{e:0816}, and \eqref{e:kern},
\begin{align}
\label{e:2376}
(\forall n\in\NN)\quad\boldsymbol{s}_n^*
&=\boldsymbol{p}_n^*+\boldsymbol{Q}\boldsymbol{p}_n
\nonumber\\
&=\big(p_{i,\overline{\delta}_i(n)}^*\big)_{i\in\BI}+
\boldsymbol{Q}\boldsymbol{p}_n
\nonumber\\
&=\big(\gamma_{i,\delta_i(n)}^{-1}\big(
x_{i,\overline{\delta}_i(n)}^*-
p_{i,\overline{\delta}_i(n)}\big)\big)_{i\in\BI}+
\boldsymbol{Q}\boldsymbol{p}_n
\nonumber\\
&=\big(\gamma_{i,\delta_i(n)}^{-1}x_{i,\overline{\delta}_i(n)}^*
\big)_{i\in\BI}-
\big(\gamma_{i,\delta_i(n)}^{-1}
p_{i,\overline{\delta}_i(n)}\big)_{i\in\BI}+
\boldsymbol{Q}\boldsymbol{p}_n
\nonumber\\
&=\boldsymbol{K}_n\widetilde{\boldsymbol{x}}_n-
\boldsymbol{K}_n\boldsymbol{p}_n.
\end{align}
Thus, \eqref{e:alg} can be recast as
\begin{equation}
\begin{array}{l}
\text{for}\;n=0,1,\ldots\\
\left\lfloor
\begin{array}{l}
\boldsymbol{p}_n=(\boldsymbol{K}_n+\lag)^{-1}
\big(\boldsymbol{K}_n\widetilde{\boldsymbol{x}}_n\big)\\
\boldsymbol{s}_n^*=\boldsymbol{K}_n\widetilde{\boldsymbol{x}}_n-
\boldsymbol{K}_n\boldsymbol{p}_n\\
\text{if}\;\scal{\boldsymbol{p}_n
-\boldsymbol{x}_n}{\boldsymbol{s}_n^*}<0\\
\left\lfloor
\begin{array}{l}
\boldsymbol{x}_{n+1}=
\boldsymbol{x}_n+
\dfrac{\lambda_n\scal{\boldsymbol{p}_n-
\boldsymbol{x}_n}{\boldsymbol{s}_n^*}}
{\|\boldsymbol{s}_n^*\|^2}\,\boldsymbol{s}_n^*\\
\end{array}
\right.\\
\text{else}\\
\left\lfloor
\begin{array}{l}
\boldsymbol{x}_{n+1}=\boldsymbol{x}_n.
\end{array}
\right.\\[2mm]
\end{array}
\right.\\
\end{array}
\end{equation}
Therefore, \cite[Theorem~4.2(i)]{Jmaa20} yields
$\sum_{n\in\NN}\|\boldsymbol{x}_{n+1}-\boldsymbol{x}_n\|^2<\pinf$.
On the one hand, in view of \cite[Lemma~A.3]{Sadd20},
we deduce from \eqref{e:1928} and \eqref{e:2488} that
$(\forall i\in\BI)$
$\boldsymbol{x}_{\delta_i(n)}-\boldsymbol{x}_n\to\boldsymbol{0}$.
On the other hand, for every $n\in\NN$, every
$\boldsymbol{x}\in\XXX$, and every $\boldsymbol{y}\in\XXX$, we
deduce from \eqref{e:2068} and the Cauchy--Schwarz inequality that
$\alpha\|\boldsymbol{x}-\boldsymbol{y}\|^2
\leq\scal{\boldsymbol{x}-\boldsymbol{y}}{
\boldsymbol{K}_n\boldsymbol{x}-
\boldsymbol{K}_n\boldsymbol{y}}
\leq\|\boldsymbol{x}-\boldsymbol{y}\|\,
\|\boldsymbol{K}_n\boldsymbol{x}-
\boldsymbol{K}_n\boldsymbol{y}\|$,
from which it follows that
\begin{equation}
\label{e:1890}
\alpha\|\boldsymbol{x}-\boldsymbol{y}\|
\leq\|\boldsymbol{K}_n\boldsymbol{x}-
\boldsymbol{K}_n\boldsymbol{y}\|.
\end{equation}
Hence,  using
\eqref{e:2138}, \eqref{e:alg}, \eqref{e:kern}, and the fact that
$\boldsymbol{Q}$ is $\chi$-Lipschitzian,
we get
\begin{align}
\label{e:1891}
\alpha^2\|\widetilde{\boldsymbol{x}}_n-\boldsymbol{x}_n\|^2
&\leq\|\boldsymbol{K}_n\widetilde{\boldsymbol{x}}_n-
\boldsymbol{K}_n\boldsymbol{x}_n\|^2
\nonumber\\
&=\big\|\big(\gamma_{i,\delta_i(n)}^{-1}\big(
x_{i,\delta_i(n)}-
\gamma_{i,\delta_i(n)}Q_i\boldsymbol{x}_{\delta_i(n)}
\big)\big)_{i\in\BI}-
\big(\gamma_{i,\delta_i(n)}^{-1}x_{i,n}-
Q_i\boldsymbol{x}_n\big)_{i\in\BI}\big\|^2
\nonumber\\
&=\sum_{i\in\BI}\big\|
\gamma_{i,\delta_i(n)}^{-1}\big(x_{i,\delta_i(n)}-x_{i,n}\big)+
\big(Q_i\boldsymbol{x}_n- Q_i\boldsymbol{x}_{\delta_i(n)}\big)
\big\|^2
\nonumber\\
&\leq\sum_{i\in\BI}2\big(\varepsilon^{-2}\big\|
x_{i,\delta_i(n)}-x_{i,n}\big\|^2+\big\|
Q_i\boldsymbol{x}_n- Q_i\boldsymbol{x}_{\delta_i(n)}\big\|^2
\big)
\nonumber\\
&\leq\sum_{i\in\BI}2(\varepsilon^{-2}+\chi^2)
\big\|\boldsymbol{x}_{\delta_i(n)}-\boldsymbol{x}_n\big\|^2
\nonumber\\
&\to 0.
\end{align}
Thus, we conclude via 
\cite[Theorem~4.2(ii) and Remark~4.3]{Jmaa20}
that $(\boldsymbol{x}_n)_{n\in\NN}$ converges weakly to a point in
$\zer\lag$, i.e., a solution to \eqref{e:9}.
Further, it is shown in the proof of
\cite[Theorem~4.2(ii)]{Jmaa20} that
$\boldsymbol{K}_n\widetilde{\boldsymbol{x}}_n
-\boldsymbol{K}_n\boldsymbol{p}_n \to\boldsymbol{0}$.
Hence, we derive from \eqref{e:1890} and \eqref{e:1891} that
$\|\boldsymbol{x}_n-\boldsymbol{p}_n\|
\leq\|\boldsymbol{x}_n-\widetilde{\boldsymbol{x}}_n\|
+\|\widetilde{\boldsymbol{x}}_n
-\boldsymbol{p}_n\|
\leq\|\boldsymbol{x}_n
-\widetilde{\boldsymbol{x}}_n\|+(1/\alpha)
\|\boldsymbol{K}_n\widetilde{\boldsymbol{x}}_n
-\boldsymbol{K}_n\boldsymbol{p}_n\|\to 0$.
\end{proof}

We are now ready to prove Theorem~\ref{t:1}.

\begin{proof}
Consider the system of monotone inclusions
\begin{equation}
\label{e:5418}
\text{find}\;\:
(\boldsymbol{x},\boldsymbol{v}^*)\in
\HHH\oplus\GGG\;\:\text{such that}\;\:
\begin{cases}
(\forall i\in I)\;\;0\in\partial\varphi_i({x}_i)+
\pnabla{i}\boldsymbol{f}_{\!i}(\boldsymbol{x})+
\sum_{k\in K}\Pi_i(\boldsymbol{L}_k^*{v}_k^*)
\\
(\forall k\in K)\;\;0\in\partial g_k^*({v}_k^*)
-\boldsymbol{L}_k\boldsymbol{x}.
\end{cases}
\end{equation}
We assume, without loss of generality, that
$I$ and $K$ are disjoint subsets of $\NN$.
Then, in view of \eqref{e:3851},
\eqref{e:5418} is a special case of \eqref{e:9} where 
$\BI=I\cup K$ and
\begin{equation}
\label{e:8718}
\begin{cases}
(\forall i\in I)\;\;\XX_i=\HH_i\;\:\text{and}\;\:
P_i=\partial\varphi_i\\
(\forall k\in K)\;\;\XX_k=\GG_k\;\:
\text{and}\;\:P_k=\partial g_k^*\\
\boldsymbol{Q}\colon(\boldsymbol{x},\boldsymbol{v}^*)
\mapsto(\boldsymbol{G}\boldsymbol{x}+
\boldsymbol{L}^*\boldsymbol{v}^*,{-}\boldsymbol{L}\boldsymbol{x}).
\end{cases}
\end{equation}
Note that $\boldsymbol{Q}$ is Lipschitzian and that,
for every $(\boldsymbol{x},\boldsymbol{v}^*)\in\HHH\oplus\GGG$
and every $(\boldsymbol{y},\boldsymbol{w}^*)\in\HHH\oplus\GGG$,
it follows from \eqref{e:Glip} that
\begin{equation}
\sscal{(\boldsymbol{x},\boldsymbol{v}^*)-
(\boldsymbol{y},\boldsymbol{w}^*)}{
\boldsymbol{Q}(\boldsymbol{x},\boldsymbol{v}^*)-
\boldsymbol{Q}(\boldsymbol{y},\boldsymbol{w}^*)}
=\scal{\boldsymbol{x}-\boldsymbol{y}}{
\boldsymbol{G}\boldsymbol{x}-\boldsymbol{G}\boldsymbol{y}}
\leq\sum_{i\in I}\chi_i\|x_i-y_i\|^2.
\end{equation}
In addition, for every $n\in\NN$ and every $k\in K_n$,
upon setting $z_{k,n}^*=y_{k,n}^*/\mu_{k,d_k(n)}$,
we deduce from \eqref{e:1494} that
\begin{equation}
z_{k,n}^*=v_{k,d_k(n)}^*+
\mu_{k,d_k(n)}^{-1}\boldsymbol{L}_k\boldsymbol{x}_{d_k(n)}
\end{equation}
and from \cite[Theorem~14.3(ii) and Example~23.3]{Livre1} that
\begin{equation}
b_{k,n}^*
=\prox_{\mu_{k,d_k(n)}^{-1}g_k^*}z_{k,n}^*
=J_{\mu_{k,d_k(n)}^{-1}P_k}z_{k,n}^*
\quad\text{and}\quad
b_{k,n}=\mu_{k,d_k(n)}(z_{k,n}^*-b_{k,n}^*).
\end{equation}
Hence, \eqref{e:1494} is a realization of \eqref{e:alg}
in the context of \eqref{e:8718} with
\begin{equation}
\label{e:7325}
\big[\;(\forall n\in\NN)\;\;\BI_n=I_n\cup K_n\;\big]
\quad\text{and}\quad
\big[\;(\forall k\in K)\;\;\chi_k=0\;\:\text{and}\;\:
\gamma_{k,n}=\mu_{k,n}^{-1}\:\big].
\end{equation}
Moreover, we observe that $\emp\neq\boldsymbol{Z}$
is the set of solutions to \eqref{e:5418}.
Hence, Proposition~\ref{p:10}\ref{p:10ii} implies that
$(\boldsymbol{x}_n,\boldsymbol{v}_n^*)_{n\in\NN}$
converges weakly to a point
$(\boldsymbol{x},\boldsymbol{v}^*)\in\boldsymbol{Z}$. By
Proposition~\ref{p:1}, $\boldsymbol{x}$ solves \eqref{e:nash}.
\end{proof}

\begin{remark}
By invoking \cite[Theorem~4.8]{Jmaa20} and arguing as in the proof
of Proposition~\ref{p:10}, we obtain a strongly convergent
counterpart of Proposition~\ref{p:10} which, in turn, yields
a strongly convergent version of Theorem~\ref{t:1}.
\end{remark}

\begin{remark}
\label{r:3}
Consider the proof of Theorem~\ref{t:1}.
We deduce from Proposition~\ref{p:10}\ref{p:10i} that
$\boldsymbol{x}_n-\boldsymbol{a}_n\to\boldsymbol{0}$
and, thus, that $\boldsymbol{a}_n\weakly\boldsymbol{x}$.
Moreover, by \eqref{e:1494}, given $i\in I$, the sequence
$(a_{i,n})_{n\in\NN}$ lies in 
$\dom\partial\varphi_i\subset\dom\varphi_i$.
In particular, if a constraint on $x_i$ is enforced via
$\varphi_i=\iota_{C_i}$, then $(a_{i,n})_{n\in\NN}$ converges 
weakly to the $i$th component of a solution $\boldsymbol{x}$
while being feasible in the sense that 
$C_i\ni a_{i,n}\weakly{x}_i$.
\end{remark}

\end{document}